\documentclass[11pt]{article}

\usepackage{amsfonts, amssymb, amsmath, amsthm,enumerate,esint,multicol, tikz}
\usepackage[colorlinks=true, pdfstartview=FitV, linkcolor=blue,citecolor=blue, urlcolor=blue]{hyperref}
\definecolor{ourRed}{RGB}{255, 81, 81}
\usepackage[shortlabels]{enumitem}
\usepackage{amsmath}
\usepackage{accents}
\usepackage{comment}
\usepackage{graphicx}
\usepackage[capitalize,nameinlink,noabbrev]{cleveref}
\usepackage{physics} 
\usepackage{setspace}
\usepackage{geometry}
\usepackage{lineno}

\numberwithin{equation}{section}

\newtheorem{Theorem}{Theorem}[section]
\newtheorem{Proposition}[Theorem]{Proposition}
\newtheorem{Observation}[Theorem]{Observation}
\newtheorem{Lemma}[Theorem]{Lemma}

\theoremstyle{definition}

\newtheorem{Example}[Theorem]{Example}

\makeatother\newtheorem{Remark}[Theorem]{Remark}

\newcommand{\B}{\mathcal{B}}

\newcommand{\cals}{\mathcal{S}}

\newcommand{\M}{\operatorname{M}}

\renewcommand{\cals}{\mathcal S}

\newcommand{\Z}{\operatorname{Z}}

\newcommand{\mr}{\operatorname{mr}}

\newcommand{\bx}{\textbf{x}}

\newcommand{\N}{\operatorname{N}}
\newcommand{\PP}{\operatorname{P}}

\title{The inverse eigenvalue problem for probe graphs} 
\author{Emelie Curl\thanks{Department of Mathematics, Statistics, and Computer Science, Hollins University, Roanoke, VA, USA (curlej@hollins.edu)}\and  
J{\"u}rgen Kritschgau\thanks{ Fariborz Maseeh Department of Mathematics and Statistics, Portland State University, Portland, OR, USA (jkritsch@pdx.edu). Research is supported by NSF RTG grant DMS-2136228.} \and Carolyn Reinhart\thanks{Department of Mathematics and Statistics, Swarthmore College, Swarthmore, PA, USA (creinha1@swarthmore.edu)} \and Hein van der Holst\thanks{Department of Mathematics and Statistics, Georgia State University, Atlanta, GA, USA (hvanderholst@gsu.edu)}
}

\begin{document}
%\linenumbers
\maketitle

\begin{abstract}
In this paper, we initiate the study of the inverse eigenvalue problem for probe graphs. A probe graph is a graph whose vertices are partitioned into probe vertices and non-probe vertices such that the non-probe vertices form an independent set. 
In general, a probe graph is used to represent the set of graphs that can be obtained by adding edges between non-probe vertices. 
The inverse eigenvalue problem for a graph considers a family of matrices whose zero-nonzero pattern is defined by the graph and asks which spectra are achievable by matrices in this family. We ask the same question for probe graphs. We start by establishing bounds on the maximum nullity for probe graphs and defining the probe graph zero forcing number. Next, we focus on graphs of two parallel paths, the unique family of graphs whose (standard) zero forcing number is two. We partially characterize the probe graph zero forcing number of such graphs and prove some necessary structural results about the family. Finally, we characterize probe graphs whose minimum rank is $0, 1, 2, n-2,$ and $n-1$.
\end{abstract}

\section{Introduction}

The inverse eigenvalue problem of a graph (IEPG) seeks to determine which spectra are achievable among all real symmetric matrices whose non-zero off-diagonal entries are determined by the edges of a graph (see \cite{MultMinor, hogben2022inverse}). 
This problem is motivated by inverse problems in the theory of vibrations \cite{Vibrations}, but has become an area of independent interest. 
Determining the achievable spectra for any given graph is a complex problem, so the focus has shifted to various sub-problems and variants. Some sub-questions involve determining the maximum eigenvalue multiplicity (or equivalently, either the maximum nullity or minimum rank) of matrices described by a graph or family of graphs. Computing the maximum multiplicity (or minimum rank), in general, remains unanswered and an active area of research (see \cite{ FH14, fallat2007minimum}).  
In this paper, we introduce a new variant of the inverse eigenvalue problem, which we call the inverse eigenvalue problem for probe graphs. To that end, we determine some bounds on maximum nullity (or minimum rank) and characterize some extreme values of maximum nullity (or minimum rank) for probe graphs.

A {\em graph} $G$ is an ordered pair $G=(V(G), E(G))$ comprising a vertex set $V(G)=\{v_1,\dots,v_n\}$ which contains vertices, $v_i$, and an edge set $E(G)$ consisting of edges, $v_i v_j$ such that $i \neq j$. 
An {\em independent } set of vertices in a graph $G$ is a subset of $V(G)$ so that no two vertices within the set form an edge in $E(G)$. 
The {\em set of all real symmetric matrices associated with a graph $G$}, $\cals(G)$, includes all $n \times n$ symmetric matrices whose $ij$th entry is 0 when $i\not=j$ and $v_iv_j\not\in E(G)$, any non-zero real number when $i\not=j$ and $v_iv_j\in E(G)$, and any real number when $i=j$. 
The  {\em inverse eigenvalue problem for graphs} asks: for a given graph $G$ and a given list of real numbers, does there exist a matrix $A \in \cals(G)$ so that the spectra of $A$ is the given list of real numbers? However, this problem has proven challenging, and several sub-problems are commonly studied. Two sub-questions of much interest comprise establishing the values of minimum rank and maximum nullity. 
The {\em minimum rank} of a graph $G$, denoted $\mr(G)$, is the minimum rank of any matrix in $\cals(G)$ and the {\em maximum nullity} of a graph $G$, denoted $\M(G)$, is the maximum nullity of any matrix $\cals(G)$. 
Note that because the diagonal entries of the matrices in $\cals(G)$ can be any real number, $\M(G)$ is also equivalent to the maximum multiplicity of an eigenvalue. 
Also, it is evident by the rank nullity theorem that $\mr(G)+\M(G)=n$, so these two parameters are studied interchangeably.

In \cite{work2008zero}, in order to provide an upper bound on maximum nullity, the authors defined a new parameter of interest, the zero forcing number of a graph $G$, denoted $\Z(G)$. 
{\em Zero forcing} is a graph coloring process by which an initial set of blue vertices eventually color the entire graph by repeated applications of a color change rule. 
The {\em (standard) zero forcing color change rule} allows a blue vertex to force a white vertex to become blue if and only if it has no other white neighbors. 
Any set of initially blue vertices that are eventually able to force the entire graph is called a {\em zero forcing set}, and the smallest size of a zero forcing set for a graph $G$ is called the {\em zero forcing number} $\Z(G)$. 
It is shown in \cite{work2008zero} that $\M(G)\leq \Z(G)$ for all graphs $G$. 
This bound is known to be tight for certain graph families, such as path and cycle graphs, but is not tight in general. 
Zero forcing has since grown into a parameter of independent interest (see, for example, \cite{FH14},\cite{H20},\cite{HHKMWY12},\cite{L19},\cite{row2012technique}).

The inverse eigenvalue problem for graphs has also been studied for several variants and restrictions. 
For example, the {\em positive semidefinite (PSD) inverse eigenvalue problem for graphs} restricts $\cals(G)$ to the set of matrices that are positive semidefinite. 
This set of matrices is denoted $\cals^+(G)$ and analogous parameters can be defined, such as $\mr^+(G)$, $\M^+(G)$, and $\Z^+(G)$ (\cite{PSD}). 
Often, variants are motivated by additional assumptions about matrices or graphs in an application. 
In our case, we are motivated by the fact that networks are not completely observed and that practitioners need to deal with all possible network structures in an unobserved portion of a network. 
Therefore, we consider probe graphs that model this situation. 

A {\em probe graph} $G^{\N}$ is a graph that contains a set of probes $\PP$ and a set of non-probes $\N$, such that $\N$ is an independent set. 
Traditionally, the study of probe graphs has focused on how edges can be added between non-probe vertices to ensure the resulting graph has a specific structure. 
For example, after edges are added, probe interval graphs are required to be an interval graph, and probe chordal graphs are required to be chordal (see \cite{mcmorris1998probe} and \cite{ProbeChordal}, respectively).
Occasionally, we will treat $G^{\N}$ as the set of graphs containing $G$ and any supergraph of $G$, say $H$, where the edges of $H$ that are not in $G$ are contained in $\N$.

We now introduce the {\em inverse eigenvalue problem for probe graphs}, which seeks to determine all possible spectra achieved by the set of matrices which is described by a probe graph. Let $G$ be a graph on $n$ vertices where $\N =\{ n-k+1, \cdots, n\}$ is an independent set of $G$.
Let $\cals(G^{\N})$ denote the set of matrices \[ A + \begin{bmatrix}O&O\\O&B\end{bmatrix}\]
where $A\in \cals(G)$ and $B$ is a real-symmetric $k\times k$ matrix. Note that this is the same as considering the matrices $\cals(H)$ for all graphs $H$, which result from adding any number of edges between vertices in $\N$.

We can now define parameters analogous to those studied in the IEPG. The {\em minimum rank of a probe graph} $G^{\N}$, denoted $\mr(G^{\N})$, is the minimum rank of any matrix in $S(G^{\N})$ and the {\em maximum nullity of a probe graph} $G^{\N}$, denoted $\M(G^{\N})$, is the maximum nullity of any matrix in $S(G^{\N})$. Note that the equality $\mr(G^{\N})+\M(G^{\N})=n$ still holds. We also define the {\em probe graph zero forcing number}, $\Z(G^{\N})$, to be the smallest set of initially blue vertices that can completely force the probe graph according to the probe graph color change rule (see Section~\ref{sec:ZF}).

We begin the paper by establishing some useful bounds on $\mr(G^{\N})$ and $\M(G^{\N})$ in Section~\ref{sec:bounds}, including showing that $\M(G^{\N})\geq |\N|$. In Section~\ref{sec:ZF}, we turn our attention to the probe graph zero forcing number, establishing the bound $\M(G^{\N})\leq \Z(G^{\N})$. %and providing an example which shows that \Z(G\N)\Z(G^{\N}) depends on not just the size of \N\N, but also its composition. 
We also study graphs on two parallel paths, providing bounds on their probe graph zero forcing number and showing that such graphs characterize the class of graphs such that $\Z(G^{\N})=2$. Section~\ref{sec:Strucpaths} provides necessary conditions for the structure of parallel paths, which are required for our characterization of the case $\mr(G^{\N})=n-2$ in Section~\ref{sec:extreme}. Turning our attention to extreme values of $\mr(G^{\N})$, in Section~\ref{sec:extreme}, we characterize probe graphs with $\mr(G^{\N})=1$, $\mr(G^{\N})= 2$, $\mr(G^{\N})=n-2$, and $\mr(G^{\N})=n-1$.

\section{Bounds on the maximum nullity and minimum rank of probe graphs}\label{sec:bounds}

In this section, we provide bounds for $\M(G^{\N})$ and $\mr(G^{\N})$. We begin with a useful bound regarding the number of non-probe vertices $|\N|$.
\begin{Proposition}\label{prop.gammabound}
For any $G^{\N}$ on $n$ vertices, \[\M(G^{\N})\geq |\N|. \]  
Equivalently, $\mr(G^{\N})\leq n-|\N|.$
\end{Proposition}

\begin{proof}
Notice that any matrix in $\cals(G^{\N})$ has the form 
\[\begin{bmatrix}A&B\\B^\top& C\end{bmatrix}\] where $A, B$ are constrained by $G^{\N}$ and $A,C$ are symmetric.
Choose $A$ to have full rank (which can always be done by translation). 
Since $A$ has full rank, a matrix $D$ exists such that $AD=B$. 
Furthermore, let $C= D^\top AD$.
Then 
\begin{align*}
\begin{bmatrix}I&O\\-D^\top &I\end{bmatrix}\begin{bmatrix}A&AD\\D^\top A& C\end{bmatrix}\begin{bmatrix}I&-D\\O &I\end{bmatrix} &= \begin{bmatrix}I&O\\-D^\top &I\end{bmatrix}\begin{bmatrix}A&O\\D^\top A& C-D^\top AD\end{bmatrix}=\begin{bmatrix}A&O\\O&O\end{bmatrix}.
\end{align*}
Thus, we have found a matrix in $\cals(G^{\N})$ with nullity $|\N|$. 
\end{proof}

% \begin{Corollary}\label{upboundongamma}
% If $G$ is a graph such that $k \geq  Z(G^{\N})$, then $|\N| \leq k$.
% \end{Corollary}
% \begin{proof}
% Follows immediately from Proposition~\ref{prop.gammabound} when $k \geq Z(G^{\N}) \geq M(G^{N}) \geq |\N|$. 
% \end{proof}

Note that making $A$ have full rank is potentially wasteful, so this bound will not always be tight. For the argument, it is only necessary that the column space of $B$ is contained in the column space of $A$. However, this bound can be tight, and we provide examples demonstrating this in Section~\ref{sec:ZF}.

In order to study the minimum rank of a probe graph, it is useful to consider the matrices $A$ and $B$ such that $[A\,\,B]$ is the submatrix whose rows correspond to probe vertices. For two matrices $A$ and $B$ with the same number of rows, we can decompose $B$ into two matrices $\text{proj}_A{B}+\text{orth}_A{B}:=B_A+B_{A^\perp}$. In other words, each column of $B_A$ is the projection of the corresponding column of $B$ onto the column space of $A$. We will use this fact to bound $\mr(G^{\N})$ in terms of the rank of $A$ and $B_{A^\perp}$. But first, we prove a useful lemma.

\begin{Lemma} \label{lem.rowspace}
If $B$ is an $n\times m$ matrix, then a symmetric $m\times m$ matrix $A$ exists such that $A$ and $B$ have the same row space.
\end{Lemma}

\begin{proof}
Suppose the row space of $B$ has dimension $r$. 
Let $\B_1=\{\bx_1,\dots, \bx_r\}$ be an orthonormal basis of row vectors for the row space of $B$. 
%Orthonormally pad this basis to $\B_2=\{\bx_1,\dots,\bx_m\}$ for the space of $m$ dimensional row vectors. 
Let $R$ be the $r\times m$ matrix where the rows are the basis vectors in $\B_1$.
Notice that the matrix $A=R^\top R$ is $m\times m$ and symmetric. 
By construction, the row space of $A$ is exactly the row space of $B$.
\end{proof}

A pattern matrix $P$ is a $\{0,*,?\}$-valued matrix, where the $0$ entries model known zero entries, $*$ entries model known non-zero entries, and $?$ model completely unknown entries.  
The pattern matrix $P(A)$ of a matrix $A$ is a $\{0,*,?\}$-valued matrix where
\[P_{ij}=\begin{cases} 
0 & \text{ if $i\neq j$ and } A_{ij}=0\\ 
*& \text{ if $i\neq j$ and } A_{ij}\neq 0\\ 
?&  \text{ if $i= j$ }
\end{cases}.\]
Notice that pattern matrices can be rectangular. 
We say a matrix $A$ realizes the minimum rank of a pattern matrix $P$ if 
\[\mr(A) = \min \{ \mr(B): P(B) = P\}.\]

\begin{Theorem}\label{thm.tophalf}
Let $P$ be the $(n-k)\times n$ pattern matrix corresponding to the probe vertices of $G^{\N}$, and let $[A\,\,B]$ be an $(n-k)\times n$ matrix that realizes the minimum rank of $P.$
Then $\mr(P)=\rank(A)+\rank(B_{A^\perp})\leq \mr(G^{\N}) \leq  \rank(A)+2\rank(B_{A^\perp}) .$
\end{Theorem}

\begin{proof}
The first inequality follows from the fact that $P$ is a sub-pattern of $G^{\N}$. 
Therefore, we will focus our attention on the second inequality.
Since every column of $B_A$ is in the column space of $A$, a matrix $C$ exists such that $B_A = AC$. 
Choose $R$ to be a $\rank(B_{A^\perp})\times k$ matrix as in the proof of Lemma~\ref{lem.rowspace}.
Since each row in $B_{A^\perp}$ is in the row space of $R^\top R$, there exists a matrix $D$ such that $DR^\top R=B_{A^\perp}$

Consider the matrix
\[Q=\begin{bmatrix}
A& AC+DR^\top R\\
C^\top A + R^\top R D^\top  & R^\top R - C^\top A C + C^\top R^\top R D^\top+DR^\top RC
\end{bmatrix}.\]
Observe that $Q\in\cals(G^{\N})$ and that $Q$ is symmetric by construction. 

Let \[T=\begin{bmatrix}I&-C\\O&I\end{bmatrix}\quad \text{and} \quad S =  \begin{bmatrix}I&-D\\O&I\end{bmatrix}.\]Now
\begin{align*}
ST^\top Q TS^\top &= S\begin{bmatrix}A & B_{A^\perp}\\ B_{A^\perp}^\top &R^\top R\end{bmatrix}S^\top= \begin{bmatrix} A-DR^\top RD^\top & O\\O & R^\top R\end{bmatrix}.
\end{align*}
Observe that $\rank(A-DR^\top RD^\top)\leq \rank(A)+ \rank(DR^\top RD^\top)$. 
Since \[\rank(DR^\top RD^\top)\leq \rank(R^\top R)\leq \rank(B_{A^\perp}),\] it follows from star similarity that  \[\rank(Q)\leq \rank(A)+2\rank(B_{A^\perp}),\] proving the theorem. 
\end{proof}

Observe that the upper bound in  Theorem~\ref{thm.tophalf} is tight when $G$ is a complete bipartite graph and $\N$ is one of the bipartite parts of $G$.
Furthermore, the lower bound in Theorem~\ref{thm.tophalf} is tight when $G=H\cup \overline{K_{\N}}$ for some graph $H$ and where $\N$ is the set of vertices in $\overline{K_{\N}}$.

\section{The probe graph zero forcing number}\label{sec:ZF}

We now turn our attention to the zero forcing number of a probe graph. We start by defining the parameter and proving some general properties. Then, we conclude by considering probe graphs $G^{\N}$ where $G$ is a graph of two parallel paths.
 
We need to define a zero forcing parameter for probe graphs that bounds the maximum nullity of all graphs in $G^{\N}$.
In other words, we want to define a zero forcing variant $\Z(G^{\N})$ such that 
\[\Z(G^{\N}) \geq \max_{H\in G^{\N}}\{\Z(H)\}\geq \max_{H\in G^{\N}}\{\M(H)\}= \M(G^{\N}).\]

To this end, we define a new color change rule. Let $G^{\N}$ be a probe graph with probe vertices $P$ and non-probe vertices $\N$. For some set of currently blue vertices $B\subseteq V(G^{\N})$, we define the \emph{probe graph color change rule} as follows:

\begin{enumerate}
\item If $v\in B\cap P$ and $v$ has exactly one white neighbor $w$, then $v$ forces $w$ to become blue. 
\item If $\N\subseteq B$, and $u\in \N$ has exactly one white neighbor $w$, then $u$ forces $w$ to become blue.
\end{enumerate}

A set of initially blue vertices $B\subseteq V(G^{\N})$ is a \emph{probe graph zero forcing set} if the entire probe graph is eventually colored blue after repeated applications of the probe graph color change rule. 
The \emph{zero forcing number of a probe graph}, $\Z(G^{\N})$, is the minimum cardinally of a probe graph zero forcing set.
Note that by definition, $\Z(G^{\N})$ is equal to the (standard) zero forcing number of $G$ with edges added between all vertices in $\N$. Another way of conceptualizing $\Z(G^{\N})$ is that a forcing set on $G^{\N}$ is a zero forcing set for all $H\in G^{\N}$, since Rule 2 always assumes the least convenient configuration of edges in $\N$. This leads to the following immediate observation.

\begin{Observation}\label{zeroiszerogam1}
If $|\N| \leq 1$, then $\Z(G) = \Z(G^{\N})$ for any graph $G$. 
\end{Observation}

Utilizing the relationship between $\Z(G^{\N})$ and $\M(G^{\N})$, we obtain the following bound on the size of $\N$ in terms of $\Z(G^{\N})$.

\begin{Proposition}\label{upboundongamma}
If $G^{\N}$ is a probe graph such that $k \geq  \Z(G^{\N})$, then $|\N| \leq k$.
\end{Proposition}
\begin{proof}
Follows immediately from Proposition~\ref{prop.gammabound} when $k \geq \Z(G^{\N}) \geq \M(G^{\N}) \geq |\N|$. 
\end{proof}

We now turn our attention to an example which demonstrates that $\Z(G^{\N})=\M(G^{\N})=\Z(G)=M(G)$ is achievable.

\begin{center}
 \begin{figure}[h]
 \begin{center}
 \includegraphics[scale=.5]{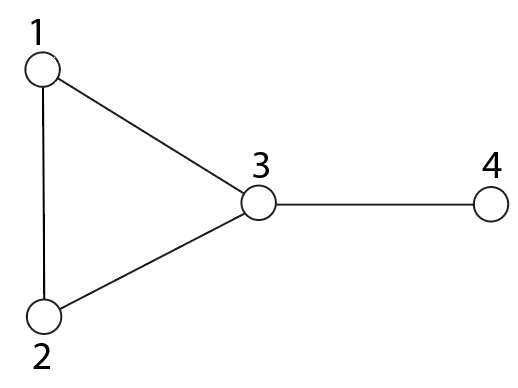}
 \end{center}
 \caption{The Paw Graph \label{fig:paw}}
 \end{figure}
 \end{center}

\begin{Example}
For the paw graph depicted in Figure~\ref{fig:paw}, we have  $\Z(\text{Paw}^{\N}) =\Z(\text{Paw})= \M(\text{Paw}^{\N})=\M(\text{Paw}) = 2$.
It was shown in \cite{MultMinor} that $\M(\text{Paw}) = 2$. Note that $\Z(\text{Paw}^{\N})\geq \M(\text{Paw}^{\N})\geq \M(\text{Paw})=2$ and $\Z(\text{Paw}^{\N})\geq \Z(\text{Paw})\geq \M(\text{Paw})= 2$. Thus, to complete the result, we demonstrate a probe graph zero forcing set of size two for the probe graph $\text{Paw}^{\N}$. Since the independence number of the Paw graph is 2, $|\N|\leq 2$. If $|\N|=1$, the set $\{1,2\}$ is a zero forcing set and shows $\Z(\text{Paw}^{\N}) =\Z(\text{Paw})=2$. If $|\N|=2$, it must be that $\N = \{1,4\}$ or $\N = \{2,4\}$. In this case, $\{2,3\}$ or $\{1,3\}$, (respectively) is a probe graph zero forcing set.
\end{Example}

To present the next result, we first require some additional definitions.
Given a probe graph zero forcing set $B$, a set of forces exists such that every vertex in $G^{\N}$ is eventually colored blue. 
Then $F_1,\dots,F_{|B|}$ is a set of {\em forcing chains} for $G^{\N}$ if each $F_i$ contains an ordered lists of vertices $v_1,\dots,v_{k}$ such that $v_j$ forces $v_{j+1}$.
Note that by definition, $B$ contains exactly the first vertex in each forcing chain. The {\em reversal} of $B$ with respect to a set of forcing chains $F_1,\dots, F_{|B|}$ is the set of vertices that are the last vertex in each forcing chain. 
Notice that a blue set $B$ may have many forcing chains and, therefore, many reversals. 
The analogous definitions for graphs were presented in \cite{BBFHHSVV10}, and the authors show that every reversal of a zero forcing set is also a zero forcing set. 
It is clear that the same result holds for the reversal of a probe graph zero forcing set.
Now, we present a necessary condition for a probe graph $G^{\N}$ to have a minimum zero forcing number.

\begin{Theorem}\label{thm:vertexcut}
If $\Z(G^{\N}) =|\N|$, then $\N$ is a vertex cut of $G$ or $\N$ is the only reversal for every probe graph forcing set which does not contain $\N$.
\end{Theorem}

\begin{proof}
Let $B$ be a minimum probe graph forcing set of $G$ which does not contain $\N$, and let $F_1,\dots, F_{|\N|}$ be a set of forcing chains of $B$. 
Notice that the forcing chains partition $V(G)$, each containing exactly one vertex from $\N$. 
Let $g_i\in \N$ be the vertex in $\N$ contained in $F_i$. 
Let $\leq$ be the partial order on $V(G)$ such that $a\leq b$ if and only if $a,b$ are in the same forcing chain and $a$ does not appear strictly after $b$ (i.e. use the linear orders from all $F_i$ to create a partial order on $V(G)$). 

\textbf{Claim:} If $xy\in E(G)$ with $x\in F_i$ and $y\in F_j$ (where $i,j$ may or may not be distinct), then either $x\leq g_i$ and $y< g_j$, or $x\geq g_i$ and $y> g_j$. 

Without loss of generality, suppose $x\leq g_i$ and $y\geq g_j$.
Since $\N$ is an independent set and our graph is loopless, at least one of the two relations must be strict.
Furthermore, if $y\geq g_j$, we know that $y$ will not be blue until every vertex in $\N$ is blue. 
Therefore, $x<g_i$ and $y>g_j$ is impossible.
In particular, if $x<g_i$, then $y=g_j$ and we are done. 
Similarly, if $y>g_j$, then $x=g_i$ and we are done. 
Thus, the claim is proven. 

Notice that an immediate consequence of the claim is that no edge passes from the vertex set under $\N$ (in the partial order) to the vertex set above $\N$. 
Since $\N$ is not the reversal of the $F_i$s, there are vertices on both sides of $\N$ in the partial order.
Therefore, $\N$ is a cut set.
\end{proof}

Notice that there is no hope for finding a converse of Theorem~\ref{thm:vertexcut}. 
The zero forcing sets for the graph $2K_n\vee \overline{K_2}$ with $\N=\overline{K_2}$ must contain at least $n+1$ vertices. 
Thus, even though $\N$ is a vertex cut, the zero forcing number is much larger than $2$.

% \subsection{Graphs of two parallel paths}\label{sec:ParallelPaths}
We now turn our attention to graphs of two parallel paths.
A graph is a \emph{graph of two parallel paths} if there exist two independent induced sub-paths $v_1\dots v_k$ and $u_1\dots u_\ell$ such that $k+\ell=n$ and if $v_iu_j\in E(G)$ then $v_xu_y$ is not an edge when $x>i, y<j$ or $x<i, y>j$.
Note that we define a graph of parallel paths to include the case when $G=P_n$. 
Furthermore, this is an alternative definition to the one given originally in \cite{johnson2009graphs}, which relies on an embedding of $G$ into the plane.
From either definition, it is relatively immediate to observe that a graph of two parallel paths is outer-planar.
Our study of this graph family is motivated by the following result:

\begin{Theorem}[Theorem 2.3 in \cite{row2012technique}] \label{Zforce2}
For a graph $G$, $\Z(G) = 2$ if and only if $G$ is a graph of two parallel paths but $G\neq P_n$.
\end{Theorem}

% \begin{Theorem}[Theorem 5.1 in \cite{johnson2009graphs}]\label{Maxnull2}
% Let $G$ be a graph. Then $\M(G) = 2$ if and only if $G$ is a graph of two
% parallel paths (but $G\neq P_n$) or $G$ is one of the types shown in Figure~\ref{SpecialGraphswithM2labeled}.
% \end{Theorem}

% \begin{figure}[h]\label{SpecialGraphswithM2labeled}
% \centering
% \includegraphics[scale=.5]{Specialgphslabeledbars}
% \caption{Special Graphs}
% \end{figure} 

First, we consider the case of paths and cycles, showing that the probe graph zero forcing number is $|\N|$.

\begin{Proposition}\label{thm:PCequalsGamma}
If $|\N| \geq 2$, then $\Z(P_n^{\N})=\M(P_n^{\N}) = |\N|$
and $\Z(C_n^{\N})=\M(C_n^{\N}) = |\N|$.  
\end{Proposition}

\begin{proof}
We will first prove $\Z(P_n^{\N})\leq |\N|$ and $\Z(C_n^{\N})\leq |\N|$ by exhibiting zero forcing sets of size $\N$.

Consider $P_n=v_1v_2\dots v_n$ first. 
If $i$ is the smallest index such that $v_i\in \N$, set $B=\{v_1\}\cup(\N\setminus \{v_i\})$. Observe that $B$ is a probe graph zero forcing set for $P_n$.

Now consider $C_n= v_1v_2\dots v_n v_1$.
Let $i$ be the smallest index with $v_i\in \N$ and let $j$ be the largest index with $v_j\in \N$. Set $B=(\N\setminus \{v_i,v_j\})\cup \{v_1,v_n\}$ and observe that $B$ is a probe graph zero forcing set for $C_n$.

Finally, we prove the lower bound. By Proposition~\ref{prop.gammabound}
\[ \M(P_n^{\N}) \geq |\N| \,\,\, \text{and} \,\,\, \M(C_n^{\N}) \geq  |\N|.\]
Since $\Z(G^{\N})\geq \M(G^{\N})$, this completes the proof.\end{proof}

We conclude this section by examining the more general case of graphs on two parallel paths. Motivated by Theorem~\ref{Zforce2}, we consider probe graphs for which $\Z(G^{\N})=2$ and show that for such graphs, $G$ must be a graph of parallel paths which is not a path graph when $|\N|\leq 1$. Note that the converse of this result will be addressed in Section~\ref{sec:extreme}. First, we require the following result about $\mr(G)$ for path graphs.

\begin{Theorem}[Theorem 2.8 in \cite{fiedler1969characterization}]
\label{PathMaxnull1}
Let $G$ be a graph. Then $\mr(G) = |G|-1$ (hence $\M(G) = 1$) if and only if $G = P_{|G|}$.
\end{Theorem}

\begin{Proposition}
If $G$ is a graph on $n$ vertices such that $2 = \Z(G^{\N})$, then $2 = \Z(G^{\N})= \M(G^{\N})$ and $G$ is a graph of parallel paths which is not a path graph when $|\N|\leq 1$.
\end{Proposition}
\begin{proof}
If $2 \geq \Z(G^{\N})$, then by Corollary~\ref{upboundongamma}, $|\N| \leq 2$.

If $|\N| \leq 1$, then by Observation~\ref{zeroiszerogam1}, $2 = \Z(G^{\N}) = \Z(G) = \M(G) = \M(G^{\N})$ and we must conclude that $G$ is a graph of parallel paths which is not itself a path graph by Theorem~\ref{Zforce2}.

If $|\N| =2$, then $2 = \Z(G^{\N}) = \M(G^{\N})$ since $\Z(G^{\N}) \geq \M(G^{\N}) \geq |\N|$. 
Notice that $\M(G^{\N})\geq \M(G)$, so $2\geq \M(G)$. If $\M(G) = 1$, then by Theorem~\ref{PathMaxnull1}, $G = P_n$. If $\M(G) =2$, then by Theorem~\ref{Maxnull2}, $G$ is a graph of parallel paths or a special graph in Figure~\ref{SpecialGraphswithM2labeled}.
However, if $G$ is one of the graphs from Figure~\ref{SpecialGraphswithM2labeled}, then $\Z(G^{\N})  \geq \Z(G) \geq 3$ by Theorem~\ref{Zforce2}.
Thus, we conclude that $G$ is a graph of parallel paths. 
\end{proof}

\section{Structure of graphs of two parallel paths}\label{sec:Strucpaths}
We completely reserve this section for the study of graphs of parallel paths and their structures. These results will be required for our conclusions in Section~\ref{sec:extreme}. As we observed while defining graphs of two parallel paths, if $G$ is a graph of two parallel paths, then $G$ is outer-planar. 
Suppose a connected graph $G$ is embedded in $\mathbb R^2$ so that the embedding is outer-planar and let $R$ be the subset of all vertices incident with some bounded face. 
We call $R$ a {\em core} of $G$. 
Alternatively, $R$ can be defined as the set of vertices in an induced cycle. 

Our first foray into understanding this family better is considering graphs with an empty core. In the case of graphs of two parallel paths, this is equivalent to the graph being a tree.

\begin{Theorem}\label{thm:emptycore}
    A tree $G$ is a graph of two parallel paths if and only if all of the following are true:
    \begin{enumerate}
    \item[1.] no vertex of $G$ has degree $4$ or greater,
    \item[2.] $G$ has at most two vertices of degree $3$, and 
    \item[3.] vertices of degree $3$ are adjacent in $G$.
    \end{enumerate}
\end{Theorem}

\begin{proof}
     It is straightforward to check that if $G$ does not satisfy any one of the enumerated conditions, then $G$ is not a graph of two parallel paths. 
    The key observation is that each of these conditions is necessary if $G$ is to be partitioned into two induced paths.

    We will now show that if $G$ satisfies the enumerated conditions, then $G$ is, in fact, a graph of two parallel paths. 
    If $G$ is a path, then we are done. 
    Furthermore, if $G$ is not connected, then $G$ is the disjoint union of two paths, and we are done. 
    Therefore, we can assume that $G$ contains at least one vertex of degree $3$ and that $G$ is connected. 
    Consider the path $P=v_0\cdots v_r$ obtained by picking an arbitrary leaf $v_0$, and choosing each sequential vertex $v_i$ such that $v_i\in N(v_{i-1})\setminus \{v_{i-2}\}$. In the case that $N(v_{i-1})\setminus \{v_{i-2}\}$ contains two vertices, take care to select $v_i$ so that $\deg(v_i)\neq 3$. 
    This will ensure that $P$ contains at most one vertex whose degree is $3$ in $G$.
    Furthermore, $P$ must contain at least one vertex of degree $3$ because $P$ is a maximal path and $G$ is connected. 

    Notice that $G-P$ is connected because $G$ does not contain a vertex of degree $4$ or greater. 
    Furthermore, $G-P$ has exactly two leaves since any vertex of degree $3$ in $V(G)\setminus V(P)$ has a neighbor in $P$. 
    Therefore, $Q=G-P$ is an induced path of $G$. 
    Now, $P$ and $Q$ demonstrate that $G$ is a graph of two parallel paths since there is exactly one edge between $V(P)$ and $V(Q)$ in $G$, which renders the ordering (embedding) condition true.
\end{proof}

After this result, what is left to consider are graphs of parallel paths with non-empty cores. The following subsection addresses their structure. 
\subsection{Structure of non-empty cores}
The following results on the structure of parallel paths with non-empty cores rest heavily on the existence of a Hamiltonian cycle in $G[R]$. 

\begin{Proposition}\label{prop:hamiltoniancycle}
    Suppose $G$ is a graph of two parallel paths.
    If $R$ is not empty, then $G[R]$ contains a Hamiltonian cycle whose edges are the boundary of the unbounded face.
\end{Proposition}

\begin{proof}
    Since $R$ is not empty, the embedding of $G[R]$ must contain a bounded face. 
    The vertices incident with any bounded face induce a cycle. 
    For the sake of contradiction, suppose $G[R]$ contains a cut vertex $x$. 
    By the definition of $R$, we know that $x$ is on a cycle $C_x$. 
    Since $x$ is a cut vertex, there is a vertex $y\in N_{G[R]}(x)\setminus V(C_x)$ which is on a cycle $C_y$. 

    Let $P$ and $Q$ be two induced paths of $G$, which realize the fact that $G$ is a graph of two parallel paths. 
    By definition of $P$ and $Q$, we have that $P\cap R$ and $Q\cap R$ forms a partition of $R$. 
    Notice that $C_x\not\subseteq P,Q$ and $C_y\not\subseteq P,Q$. 
    This implies that $P$ and $Q$ each contain at least one vertex from $C_x$ and one vertex from $C_y$. 
    Since $C_x$ and $C_y$ are separated by a cut vertex, this is a contradiction. 

    Now $G[R]$ is an outer-planar graph with no cut vertex such that every vertex is on the boundary of a bounded face. 
    Consider an outer-planar embedding of $G[R]$, and order the vertices $c_0,\dots, c_r$ by traversing the boundary of the unbounded face. 
    If there exists $c_i=c_j$ for $0\leq i< j \leq r$, then we have found a cut vertex. 
    This implies that the boundary of the unbounded face is a cycle.
\end{proof}

Notice that if $G[R]$ has two distinct Hamiltonian cycles, then $G[R]$ contains $K_4$ as a minor. 
Thus, with the assumption that $G$ is outer-planar, it follows that the Hamiltonian cycle of $G[R]$ is unique. 
In the case that $G[R]$ has a Hamiltonian cycle $C$ whose edges are the boundary of the unbounded face, we will call $C$ the \emph{outer-cycle of $G$}. 
We say an edge of $G[R]-E(C)$ is an \emph{interior edge}. 
Furthermore, interior edges do not cross in the following sense: 
if $C=c_0c_1\dots c_{r-1}c_0$ and $c_xc_y$ $c_uc_v$ are interior edges with $0\leq x,y,u,v\leq r$, then without loss of generality, $x\leq u$ and $v\leq y$. 
This again follows from the fact that $G$ contains no $K_4$ as a minor. 
One consequence of this line of reasoning is that any outer-planar embedding of $G[R]$ is combinatorially unique, where the boundary of the outer face is always given by $C$, and the relative positions of interior edges are preserved for all outer-planar embeddings of $G[R]$.\\
\indent What follows in this section are several auxiliary results needed in an effort to better understand the relationship between the outer-cycle and any interior edges of graphs of parallel paths with non-empty cores.

\begin{Lemma}\label{lemma:partition}
    Suppose $G$ is a graph of two parallel paths realized by $P$ and $Q$. 
    If the outer-cycle of the core $R$ is given by $C=c_0c_1\dots c_{r-1}c_0$, then $V(P)\cap V(R)$ and $V(Q)\cap V(R)$ are consecutive intervals and $P[V(P)\cap V(R)]$ and $Q[V(Q)\cap V(R)]$ only use edges in $C$.
\end{Lemma}

% \textcolor{red}{JK: we discussed whether change this to a observation... case to be made either way. No one cares.}

\begin{proof}
    Since $G$ is a graph of two parallel paths with a non-empty core, there must be a non-trivial partition of the vertices of $G$ into two induced paths $P$ and $Q$.
    Let $C=c_0c_1\dots c_{r-1}c_0$ be the outer-cycle of $G$, obtained from Proposition~\ref{prop:hamiltoniancycle}.
    Let $f$ be the unbounded face of an outer-planar embedding of $G$. 
    Then $X=\mathbb R^2\setminus f$ is the closed and bounded region corresponding to $G[R]$ whose boundary is $C$.
    Notice that if $P\cap C=C$, then $P$ is not an induced path. 
    Therefore, $C\setminus P$ is non-empty.
    In particular, assume that $c_{r-1}\notin P$. 
    Furthermore, if $G[R]$ is a cycle, then we are done.
    
    We claim that (without loss of generality) $P\cap C$ consists of consecutive vertices on $C$ modulo $r$.
    To prove this claim, suppose for the sake of contradiction that there exist indices $0\leq i< j<k<r-1$ such that $c_i,c_k\in P$ but $c_j\notin P$. 
    Since $P$ is a path, there is a path contained in the interior of $X$ from $c_i$ to $c_k$. 
    Notice that this path separates $c_j$ from $c_{r-1}$ by the Jordan Curve Theorem. 
    This contradicts the fact that $Q$ is a path. 
    Thus, $P\cap C$ consists of consecutive vertices of $C$, and by a symmetric argument, the same holds for $Q\cap C$. 

    Furthermore, since $P$ is an induced path, the edges of $P[P\cap C]$ are contained in the edges of $C$. 
    In particular, neither $P$ nor $Q$ contains an edge that is not a boundary edge of the unbounded face. 
\end{proof}

Provided $G[R]$ is not a cycle, we have a useful description of its composition.

\begin{Lemma}\label{lemma:interiorface}
    If $G$ is a graph of two parallel paths with a non-empty core $R$ and outer-cycle $C$, then every induced cycle of $G$ has at most two interior edges. 
\end{Lemma}

\begin{proof}
    For the sake of contradiction, suppose that $G$ contains an induced cycle $C'$, which contains at least three interior edges.
    By definition of the core, $V(C')\subseteq V(C)$. 
    Notice that $C'$ has at least three vertices, and due to the outer-planar embedding, the vertices of $C'$ must respect the cyclic order of $C$; otherwise, there is a pair of crossing edges in the embedding of $G$. 
    In particular, $C'=c_{i_0}\cdots c_{i_j}$ with $0\leq i_0<i_1<\cdots i_j<r$. 
    Let $c_{x_1}c_{x_2}, c_{y_1}c_{y_2}$, and $c_{z_1}c_{z_2}$ be three interior edges of $C'$ so that $i_0\leq x_1<x_2\leq y_1< y_2\leq z_1<z_2$ with the possibility that $x_1=z_2$. 
    By definition of interior edges $x_1+1< x_2$, $y_1+1< y_2$, and $z_1+1<z_2$.  
    Furthermore, $P\cap C'$ and $Q\cap C'$ must partition $C'$ into two non-empty subsets, which are consecutive intervals in $C$ by Lemma~\ref{lemma:partition}. 
    By the Pigeon Hole Principle, we may assume that of $c_{x_1}, c_{y_1}, c_{z_1}$ say $c_{x_1},c_{y_1}\in P$ while $c_{y_2}\in Q$. 
    Then $P$ must contain all the vertices between $c_{x_1}$ and $c_{y_1}$ on $C$, forming a cycle in $V(P)$. 
    This contradicts the fact that $P$ is an induced path. 
\end{proof}

\begin{Lemma}\label{lemma:endface}
    Suppose $G$ is a graph of two parallel paths with a non-empty core $R$. 
    If $G[R]$ is not a cycle, then $G[R]$ contains exactly two induced cycles with exactly one interior edge.
\end{Lemma}

\begin{proof}
    Since $G[R]$ is not a cycle, it contains at least one interior edge.
    In particular, there is at least one induced cycle in $G[R]$, which contains exactly one interior edge.
    Let $P$ and $Q$ be the induced paths that realize the fact that $G$ is a graph of two parallel paths. 
    Let the outer-cycle be denoted by $C = c_0 c_1 \dots c_{r-1}c_0$.
    
    Suppose $C^1=c_0\dots c_jc_0$ is an induced cycle of $G[R]$ with exactly one interior edge given by $c_0c_j$ (with $j<r-1$).
    Without loss of generality, $P\cap C=c_i\dots c_k$ with $0<i\leq j\leq k\leq r-1$ and $Q\cap C=c_{i-1}c_{i-2}\dots c_{k+1}$. 
    Let $x$ be the largest index such that $c_x\in P$ is incident with an interior edge, say $c_xc_y$ where $y$ is selected to minimize the distance to $c_{k+1}$ along $Q$.  
    Notice $i\leq x<y\leq r$.
    Since $G$ is outer-planar, $C^2=c_x\dots c_yc_x$ is an induced cycle with the only interior edge being $c_xc_y$. 
    Notably, $C^2$ does not contain the edge $c_0c_1$, so $C^1\neq C^2$.
    In particular, $G$ contains at least two induced cycles with exactly one interior edge.
    
    For the sake of contradiction, suppose that $C^1, C^2,$ and $C^3$ are induced cycles of $G[R]$ with exactly one interior edge. 
    Once again, using the outer-cycle $C=c_0c_1\dots c_{r-1}c_0$, we may assume that the induced cycles are disjoint intervals of $C$ given by $C^i=c_{x_i}\cdots c_{y_i}c_{x_i}$ with interior edge $c_{x_i}c_{y_i}$ such that $x_i+1< y_i$.
    Without loss of generality, $0\leq x_1< y_1\leq x_2< y_2\leq x_3< y_3\leq r$ with the possibility that $c_{x_1}=c_{y_3}$. 
    Without loss of generality, $x_1,x_2\in P$ and $x_3\in Q$ by the Pigeon Hole Principle. 
    Since $P$ must contain an interval of $C$ by Lemma~\ref{lemma:partition}, we have $C^1\subseteq P$. 
    This contradicts the fact that $P$ is an induced path. 
\end{proof}

\begin{Lemma}\label{lem:insertion}
    Suppose $G$ is a graph of two parallel paths partitioned by $P$ and $Q$ with non-empty core $R$. 
    If $G$ contains exactly two induced cycles with one interior edge $C^1, C^2$, then at least one endpoint of $P[V(P)\cap V(R)]$ and at least one endpoint of $Q[V(Q)\cap V(R)]$  is contained in $C^i$ for $i=1,2$. 
\end{Lemma}

\begin{proof}
    By Lemma~\ref{lemma:partition}, $P[V(P)\cap V(R)]$ and $Q[V(Q)\cap V(R)]$ partitions $R$ into consecutive intervals of the outer-cycle $C$. 
    Since $P$ and $Q$ are induced paths, we have that $C^1,C^2\not\subseteq P$ and $C^1,C^2\not\subseteq Q$. 
    Thus, $P[V(P)\cap V(R)]$ and $Q[V(Q)\cap V(R)]$ each have one endpoint in $C^1$ and one endpoint in $C^2$. 
\end{proof}

\begin{Lemma}\label{lem:insertion.points}
    Suppose $G$ is a graph of two parallel paths partitioned by $P$ and $Q$ with non-empty core $R$. 
    Then $G-R$ is the disjoint union of at most $4$ paths $\{P^i\}_{i=1}^k$ with $k\leq 4$ such that there is exactly one edge $x_iy_i$ where $x_i$ is a leaf of $P^i$ and $y_i\in R$.
\end{Lemma}

\begin{proof}
    By definition of $R$, it follows that $G-R$ is a forest, and each component $G-R$ is connected to $R$ via exactly one edge (otherwise, some vertex in $G-R$ would be in a cycle of $G$). 
    Let the components of $G-R$ be enumerated as $\{P^i\}_{i=1}^k$ and let $x_iy_i$ be the edge with $x_i\in P^i$ and $y_i\in R$. 

    For the sake of contradiction, suppose that $G[V(P^i)\cup \{y_i\}]$ contains a vertex $v$ of degree $3$ or greater. 
    Notice that $G[V(P^i)\cup \{y_i\}]$ is a tree with at least three leaves $\ell_1,\ell_2,y_i$. 
    Without loss of generality, suppose $v,y_i,\ell_1\in P$.
    Notice that $\ell_2\in Q$ since $P$ is an induced path. 
    Furthermore, since $G[R]$ contains a cycle, some vertex $u$ in $R$ is contained in $Q$. 
    This implies that $Q$ contains a sub-path from $\ell_2$ to $u$ since $Q$ is an induced path.
    However, this implies that $v\in Q$, which is a contradiction. 
    Therefore, each $G[V(P^i)\cup\{y_i\}]$ is a path. 

    Notice that each $P^i$ contains a pendant of $G$. 
    Since $P$ and $Q$ can each contain at most two pendants of $G$, it follows that $k\leq 4$.
\end{proof}

In light of Lemma~\ref{lem:insertion.points}, we introduce some terminology. 
If $x_iy_i$ is the edge from $P^i$ to $G[R]$ as in the statement of Lemma~\ref{lem:insertion.points}, then we call $y_i$ an \emph{insertion point} of $G$. A key branching point in our analysis of graphs of two parallel paths is whether all insertion points are distinct. 

\begin{Proposition}\label{prop:alpha.case}
    Suppose $G$ is a graph of two parallel paths partitioned by $P$ and $Q$ with non-empty core $R$. Let $G-R$ be the disjoint union of at most $4$ paths $\{P^i\}_{i=1}^k$ with $2\leq k\leq 4$ such that there is exactly one edge $x_iy_i$ where $x_i$ is a leaf of $P^i$ and $y_i\in R$. 
    If $y_1=y_2$, then every interior edge of $G$ is incident to $y_1$, and $y_3,y_4 \in N(y_1)$. 
\end{Proposition}

\begin{proof}
By Lemma~\ref{lemma:partition}, $P$ and $Q$ each contain a vertex from $R$. 
Suppose that $y_1c_1\dots c_{r-1}$ is the outer-cycle of $R$ and that $P$ contains $y_1$. 
Since $y_1$ is a cut vertex separating $P^1$ and $P^2$ from $R$, it follows that $V(P^1)\cup V(P^2)\cup\{y_1\}=V(P)$. 
Then $Q$ contains $c_1\dots c_{r-1}$ as an induced sub-path. Thus, $y_3,y_4$ (if the corresponding components exist) must be $c_1$ and $c_{r-1}$. 
Since $Q$ is an induced path, all internal edges of $R$ must contain $y_1$. 
\end{proof}

In light of Proposition~\ref{prop:alpha.case}, we have a good understanding of the structure of graphs of two parallel paths with non-empty cores and repeated insertion points $y_1=y_2$ of the pendant paths. 
These graphs must look like some graph depicted in Figure \ref{fig:alpha}, which we call $\alpha$-graphs for convenience. 

\begin{Observation}\label{obs:alpha.unique}
    The path partition $P$ and $Q$ of an $\alpha$-graph, which is a graph of two parallel paths, is unique, as the repeated insertion point $y_1=y_2$ forces $V(P^1)\cup V(P^2)\cup\{y_1\}=V(P)$. 
\end{Observation}

\begin{figure}
    \centering
    \includegraphics[scale=.8]{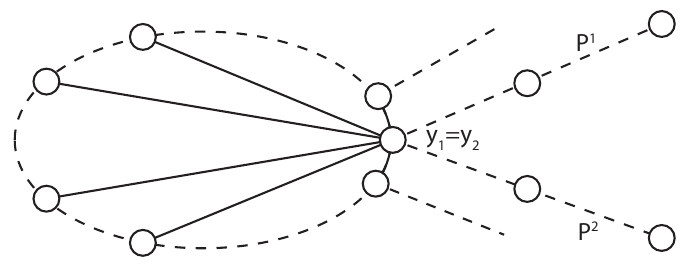}
    \caption{For $r\geq 3$, $p_1,p_2\geq 1$, $p_3,p_4\geq 0$, let $G(r,p_1,p_2,p_3,p_4)$ denote the set of graph of two parallel paths with an outer-cycle $\{c_0,\dots,c_{r-1}\}$, possible interior edges $c_0c_i$ for $i\neq r-1,0,1$, and pendent paths of length $p_1,p_2,p_3,p_4$ with insertion points $y_1=y_2=c_0$, $y_3=c_{r-1}$, and $y_4 = c_1$.
    An $\alpha$-graph is a graph that is a member of  some $G(r,p_1,p_2,p_3,p_4)$ with appropriate parameters.}
    \label{fig:alpha}
\end{figure}

\begin{Proposition} \label{prop:spokes}
Suppose $G$ is a graph of two parallel paths partitioned by $P$ and $Q$ with non-empty core $R$. Let $G-R$ be the disjoint union of at most $4$ paths $\{P^i\}_{i=1}^k$ with $k\leq 4$ such that there is exactly one edge $x_iy_i$ where $x_i$ is a leaf of $P^i$
 and $y_i\in R$ are distinct.  
    
If $G[R]$ is not a cycle, then 
    \begin{enumerate}
        \item each $y_i$ is contained in an induced cycle of $G$ with exactly one interior edge,
        \item there is an assignment $\phi:\{y_i\}_{i=1}^k\to\{ C^1,C^2\}$ so that $|\phi^{-1}(C^i)|\leq 2$, $y_i\in \phi(y_i)$,  $\phi(y_i)=\phi(y_j)$ only if $y_i$ and $y_j$ are consecutive on the outer-cycle of $G[R]$. 
    \end{enumerate}
    
    Moreover, if $G[R]$ is a cycle and $k\geq 3$, then  $G[\{y_i\}_{i=1}^k]$ induces one of the following subgraphs: $K_2\cup K_1$, $2K_2$, $P_3$, $P_4$, $C_3$, or $C_4$. 
\end{Proposition}

\begin{proof}
    We will proceed with the assumption that $G[R]$ is not a cycle. 
    Let $C=c_0c_1\dots c_{r-1}c_0$ denote the outer-cycle of $G$. 
    Claims 1. and 2. follow from Lemma~\ref{lem:insertion}. 

    The case when $G[R]$ is a cycle follows similarly. 
\end{proof}

We conclude this section with a characterization of graphs of parallel paths $P$ and $Q$ with non-empty cores.

\begin{Theorem} \label{thm:non-emptycore}
    Suppose $G$ is a connected graph with a cycle. 
    Then $G$ is a graph of two parallel paths if and only if $G$ is outer-planar with non-empty core $R$ such that 
    \begin{enumerate}
        \item[1.] the boundary of the unbounded face of $G[R]$ is a Hamiltonian cycle,
        \item[2.] every induced cycle of $G$ has at most $2$ interior edges,
        \item[3.] if $G[R]$ is not a cycle, then exactly two induced cycles have exactly one interior edge, and
        \item[4.] $G-R$ satisfies the conditions of Proposition~\ref{prop:spokes} or $G$ is an $\alpha$-graph;
    \end{enumerate}
\end{Theorem}

\begin{proof}
For the forward direction, notice that 1.- 4. are proven by Proposition~\ref{prop:hamiltoniancycle}, Lemma~\ref{lemma:interiorface}, Lemma~\ref{lemma:endface}, Proposition~\ref{prop:alpha.case} and Proposition~\ref{prop:spokes} respectively. 

We will now focus on constructing a pair of induced paths $P$ and $Q$ in $G$ which will show that $G$ is a graph of two parallel paths. 
First, if 1.-4. hold and $G$ is an $\alpha$-graph, then $P = P^1\cup\{c_0\}\cup P^2$ and $Q= G-P$ gives a path partition of $G$.

Assume that $G[R]$ is not a cycle.
Let $C=c_0c_1\dots c_{r-1}c_0$ denote the outer-cycle of $G$. 
Let $C^1= c_0\dots c_xc_0$ and $C^2=c_y\dots c_zc_y$ be the two induced cycles with exactly one interior edge where $0<x\leq y <z\leq r$. 
Notice that any interior edge $c_uc_v$ has $x\leq u\leq y$ and $z\leq v\leq r$.
Let $P' = G[\{c_x, \dots, c_y\}]$ and $Q'=G[\{c_z,\dots, c_{r-1},c_0\}]$, which are both induced paths. 
Notice that $P'$ and $Q'$ can be extended to induced paths that partition $G$ using the conditions on $G-R$. 
In particular, each $y_i$ is a vertex on $\phi(y_i)$, informing which end of $P'$ or $Q'$ is extended to include $y_i$ and $P^i$.  
The case when $G[R]$ is a cycle follows similarly. 
This concludes the proof.
\end{proof}

\subsection{Applying characterizations of graphs of two parallel paths}

Suppose $G$ is a tree of two parallel paths.
The characterization theorems let us determine the sets $\N$ for which  $\M(G^{\N})=2$. 
We only need to consider the case when $|\N|=2$, since otherwise, we have $\M(G^{\N})\geq 3$ by Proposition~\ref{prop.gammabound}.
Notice that  $G^{\N}$ contains two graphs, namely, $G$ and a unicyclic graph $G^+$, which is the result of adding edges to $G$ within the set $\N$.
Since $G$ is a graph of two parallel paths by assumption, we must consider the possibilities for $G^+$. 
Notice that $G^+$ is not a graph of the forms featured in Figure~\ref{SpecialGraphswithM2labeled} since these graphs contain too many pendants or too many cycles.
The other possibility is that $G^+$ is a graph of two parallel paths.
Since $G^+$ is unicyclic, it will be a graph of two parallel paths if and only if it satisfies the conditions of Theorem~\ref{thm:non-emptycore}. 
Clearly, $G^+$ is outer-planar with a non-empty core $R$ such that $G[R]$ is a cycle and, therefore, has a Hamiltonian cycle. 
Since $G^+$ is unicyclic, it does not have interior edges, which render Conditions 2. and 3. vacuously satisfied. 
However, care needs to be taken for $G^+$ to satisfy Condition 4.

\begin{Proposition}\label{prop:tree.uniquness}
Suppose $G$ is a tree of two parallel paths, and $S$ is the subset of vertices $G$ with degree $3$. 
Let $\N$ be a pair of non-adjacent vertices in $G$. 
Then $G^+$ is a graph of two parallel paths if and only if $|S|\leq 1$ or $\N$ contains vertices from two different components of $G-e$ where $e$ is the edge between vertices of degree $3$ (and $\N\neq S$). 
\end{Proposition}

\begin{proof}
    First, consider the backward direction. 
    If $S$ is empty, then $G$ is a path. 
    In this case, $G^+$ is a graph of two parallel paths by inspection. 
    Furthermore, if $|S|=1$ and $G$ is a tree with one vertex of degree $3$, then $G^+$ is also a graph of parallel paths by inspection. 
    On the other hand, if $S$ is not empty, then we have that $\N$ contains vertices from two different components of $G-e$. 
    This implies that $S$ is contained in the cycle of $G^+$. 
    Thus, $G[R]$ is a cycle, and $G[R]$ contains at most $4$ components satisfying Condition 4. of Theorem~\ref{thm:non-emptycore}.

    Let us now turn to the forward direction. 
    If $|S|\leq 1$, then we are done. 
    Therefore, we will assume that $|S|\geq 2$. 
    By Theorem~\ref{thm:emptycore}, $G-S$ has 2 components. 
    By Theorem~\ref{thm:non-emptycore}, $G^+$ is a graph of two parallel paths only if $G-R$ is a disjoint union of at most $4$ paths for which all vertices in $G-R$ have degree at most two in $G$. 
    This implies that $R$ must contain $S$, which is to say, all vertices in $S$ must be contained in the cycle created by adding the edge between the vertices in $\N$ to $G$.
    This is only possible if $\N$ contains at least one vertex from two different components of $G-e$.
\end{proof}

If $G^+$ is a graph of two parallel paths, then $G^+$ must be outer-planar. 
This alone already restricts the possibilities of $\N$. 
Let $d^*(x,y)$ in $G$ be the number of edges in the Hamilton cycle of $R$ on a shortest path between $x$ and $y$, which does not use interior edges. Note that the path used to compute $d^*(x,y)$ may contain edges in a pendant path, but these edges do not contribute to the total value.

Recall that the outer-cycle $C$ is unique and fixes a cyclic ordering of the vertices of $R$ up to an orientation of the cycle. 
Suppose $c_0,\dots, c_{r-1},c_0$ is an orientation of $C$. 
Let $C[x,y]$ be the set of vertices that succeed $x$ and precede $y$ in the cyclic ordering (including $x,y$). 
With this notation, $C[y,x]$ is the set of vertices that precede $x$ and succeed $y$. 
Furthermore, we will use $C(x,y)$ to denote the set of vertices that strictly succeed $x$ and strictly precede $y$. 

\begin{Lemma}\label{lem:outerplanar.conditions}
    Let $\N=\{x,y\}$ and $G$ be a graph of two parallel paths with non-empty core $R$. 
    The graph $G^+$ is outer-planar if and only if either 
    \begin{enumerate}
        \item $\N\subseteq R$ such that no interior edge of $G$ contains one endpoint in $C(x,y)$ and the other endpoint in $C(y,x)$; or
        \item without loss of generality, $x$ is on a pendant path $P^i$ and $d^*(x,y) \leq  1$. 
    \end{enumerate}
\end{Lemma}

\begin{proof}
    Suppose $\N\subseteq R$. 
    Recall that the core $R$ has a unique Hamilton cycle whose edges form the boundary of the outer-face of $R$. 
    Therefore, $R^+$ is outer-planar if and only if the new edge $xy$ does not cross any existing edge, which is to say that $G$ does not contain an interior edge with one endpoint in $C(x,y)$ and one endpoint in $C(y,x)$. 

    Now suppose $x\in P^i$ and $d^*(x,y)\leq 1$. 
    Notice that $G^+$ is still outer-planar using the same embedding that is used for $G$ (as depicted in Figure~\ref{fig:outerplanner}) since the new edge $xy$ does not separate any vertex from the unbounded face. 
    %Then $G^+$ does not have a $K_{2,3}$ or $K_4$ as a minor since $\{y^i\}= N(P^i)$ is a cut vertex separating $x$  from $V(G)\setminus V(P^i)$ and $N[V(P^i)]$ has at most $1$ vertex of degree $3$ or greater.\carolyn{see Figure~\ref{fig:outerplanner}}

    On the other hand, if $d^*(x,y)\geq 3$, then we obtain a $K_{2,3}$ minor where $\{y,y^i\}$ is the part of size $2$, and the part of size $3$ is obtained using $x$ and a vertex from each path from $y$ to $y^i$ on the Hamilton cycle of $R$.
\end{proof}

\begin{figure}
    \centering
    \includegraphics[scale=.8]{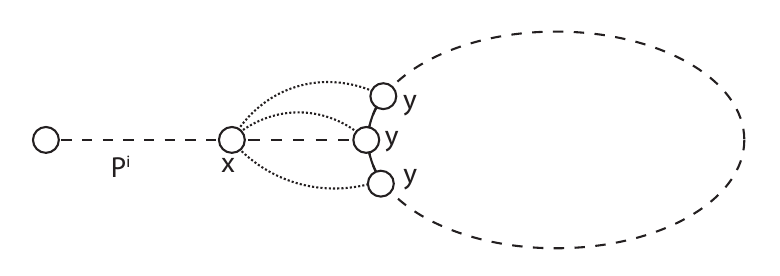}
    \caption{A graph of two parallel paths with the vertex $x$ and three potential $y$ vertices. The dotted edges are potential edges $xy$.}
    \label{fig:outerplanner}
\end{figure}

We now observe a necessary condition for $G^{+}$ to be a graph of two parallel paths, which will be useful for the following characterizations.
 
\begin{Observation}\label{obs:core.conditions}
Let $\N=\{x,y\}$.
If the graph $G^+$ is a graph of two parallel paths, then  $G[R]^+$ contains at most two induced cycles with exactly one interior edge, and every induced cycle has at most 2 interior edges, Condition 2 in Proposition~\ref{prop:spokes} holds for $G^+$, and if $x$ is on a pendant path then $d^*(x,y)=1$.  
\end{Observation}

Given that $G$ is a graph of two parallel paths with a non-empty core, we can characterize when $G^{+}$ is a graph of two parallel paths as well. Our characterization is broken into pieces, depending on if $G$ has distinct insertion points, whether or not $G[R]$ is a cycle, and on the location of $\N$. First, we consider the case where $G$ has distinct insertion points, $G[R]$ is not a cycle, and $\N$ is contained in the core $R$.

\begin{Proposition}\label{prop:mess1}
Suppose $G$ is a graph of two parallel paths with a non-empty core $R$ such that $G[R]$ is not a cycle and $G$ has distinct insertion points. 
Let $P^i$ be the pendant paths of $G$ with insertion points $y_i$.
Let $\N=\{x,y\}$ be a pair of non-adjacent core vertices in $G$. 
Then $G^+$ is a graph of two parallel paths if and only if 
\begin{enumerate}
    \item no interior edge of $G$ contains one endpoint in $C(x,y)$ and the other endpoint in $C(y,x)$; 
    \item $G[R]^+$ contains at most two induced cycles $C^1, C^2$ with exactly one interior edge and every induced cycle as at most two interior edges; and 
    \item there is an assignment $\phi:\{y_i\}_{i=1}^k\to\{ C^1,C^2\}$ so that $|\phi^{-1}(C^i)|\leq 2$, $y_i\in \phi(y_i)$,  $\phi(y_i)=\phi(y_j)$ only if $y_i$ and $y_j$ are consecutive on the outer-cycle of $G[R]^+$.
\end{enumerate}
\end{Proposition}

\begin{proof}
    If $G^+$ is a graph of two parallel paths, then Statement 1.-3. follow from Lemma~\ref{lem:outerplanar.conditions} and Observation~\ref{obs:core.conditions}.

    Assume that Statement 1.-3. are satisfied by $G^+$.
    By Statement 1. and Lemma~\ref{lem:outerplanar.conditions}, we have that $G^+$ is outer-planar and Condition 1. of Theorem~\ref{thm:non-emptycore} is satisfied by $G^+$. 
    Statements 2. and 3. and the fact that $G$ is outer-planar imply Conditions 2.-4 of Theorem~\ref{thm:non-emptycore} are satisfied. 
    Thus, $G^+$ is a graph of two parallel paths.
\end{proof}

Next, we consider the case where $G$ has distinct insertion points, $G[R]$ is not a cycle, and at least one vertex in $\N$ is not in the core $R$.

\begin{Proposition}\label{prop:mess2}
Suppose $G$ is a graph of two parallel paths with a non-empty core $R$ and distinct insertion points such that $G[R]$ is not a cycle. 
Let $P^i$ be the pendant paths of $G$ with insertion points $y_i$.
Let $\N=\{x,y\}$ be a pair of non-adjacent vertices in $G$ with $x$ on a pendant path $P^1$. 
Then $G^+$ is a graph of two parallel paths if and only if 
\begin{enumerate}
    \item[1a.] $y\in N[P^2]$ where $y_1$ is adjacent to $y_2$ on the outer-cycle of $G$; and 
   \item[2a.] there is an assignment $\phi:\{y_i\}_{i=1}^k\to\{ C^1,C^2\}$ (the cycles with exactly one interior edge in $G$) so that $|\phi^{-1}(C^i)|\leq 2$, $y_i\in \phi(y_i)$,  $\phi(y_i)=\phi(y_j)$ only if $y_i$ and $y_j$ are consecutive on the outer-cycle of $G[R]$, and $\phi(y_1)=\phi(y_2)$; 
\end{enumerate}
or
\begin{enumerate}
    \item[1b.] $y$ is adjacent to $y_1$ on the outer-cycle of $G$; and 
    \item [2b.]  there is an assignment $\phi:\{y_i\}_{i=1}^k\to\{ C^1,C^2\}$ (the cycles with exactly one interior edge in $G$) so that $|\phi^{-1}(C^i)|\leq 2$, $y_i\in \phi(y_i)$,  $\phi(y_i)=\phi(y_j)$ only if $y_i$ and $y_j$ are consecutive on the outer-cycle of $G[R]$, and $|\phi^{-1}(y_1)|=1$.
\end{enumerate}
\end{Proposition}

\begin{proof}
    Assume that $G^+$ is a graph of two parallel paths. 
    By Lemma~\ref{lem:outerplanar.conditions}  and Observation~\ref{obs:core.conditions} we have that either Statement 1a. or 1b. is satisfied, as $d^*(x,y)= 1$. 

    Assume we are in the case of 1a. and that the only assignments $\phi$ for $G$ have $\phi(y_1) \neq \phi(y_2)$. 
    Since $y_1$ and $y_2$ are adjacent on the outer-cycle of $G$, it follows that $y_1$ and $y_2$ are both contained in a $C^1$. 
    If no other insertion point is in $C^1$, then we could have $\varphi(y_1)=\varphi(y_2)$ satisfying Statement 2a. 
    Therefore, we will assume that $y_3\in C^1$. 
    However, this implies that the insertion point $y_3$ does not lie on an induced cycle of $G^+$ with exactly one interior edge since $C^1$ will have two interior edges in $G^+$. 
    This contradicts the assumption that $G^+$ is a graph of two parallel paths. 

    Assume we are in the case of 2b. and that the only assignments $\phi$ for $G$ have $|\phi^{-1}(y_1)|=2$. 
    This implies that $y_1,y_2\in C^1$ for some cycle $C^1$ with exactly one interior edge. 
    Furthermore, $y\not\in N[P^2]$ and $y$ is not an insertion point since otherwise we would be in case 1a.
    Therefore, $y$ and $y_1$ must be in $C^2$ since otherwise, $G^+$ has too many cycles with exactly one interior edge. 
    Thus, either $C_2$ contains only one insertion point (namely $y_1$) or $C_2\cap C_3$ contains $y_3$ and $y_3$ is not adjacent to $y_2$ on the outer-cycle of $G$. 
    In either case, $G^+$ fails Condition 4. of Theorem~\ref{thm:non-emptycore}.
    This concludes the proof of the forward direction. 

    Assume 1a. and 2a. are true and that $\phi$ demonstrates that 2a. holds. 
    Notice that the insertion points of $G^+$ can be associated with the insertion points of $G$ according to leaves of the pendant paths (where we have an insertion point for an empty path if $x$ is a leaf of $G$). 
    Then the partition $\varphi$ of the insertion points of $G^+$ that follows the partition $\phi$ satisfies Condition 4. of Theorem~\ref{thm:non-emptycore}. 
    Furthermore, Conditions 1.-3. of Theorem~\ref{thm:non-emptycore} follow from 1a. 
    Therefore, $G^+$ is a graph of two parallel paths. 

    Assume 1b. and 2b. are true and that $\phi$ demonstrates that 2a. holds. 
    Then the partition $\varphi$ of the insertion points of $G^+$ that follows the partition $\phi$ satisfies Condition 4. of Theorem~\ref{thm:non-emptycore}. 
    Furthermore, Conditions 1.-3. of Theorem~\ref{thm:non-emptycore} follow from 1a. 
    Therefore, $G^+$ is a graph of two parallel paths. 
\end{proof}

Now we consider the case where $G$ has distinct insertion points, $G[R]$ is a cycle, and $\N$ is contained in the core.

\begin{Proposition}\label{prop:cycle.mess1}
    Suppose $G$ is a graph of two parallel paths with a non-empty core $R$ and distinct insertion points such that $G[R]$ is a cycle. 
    Let $P^i$ be the pendant paths of $G$ with insertion points $y_i$ and let $0\leq k \leq 4$ be the number of pendant paths.
    Let $\N=\{x,y\}\subseteq R$. 
    Then $G^+$ is a graph of two parallel paths if and only if 
    \begin{enumerate}
        \item $k=0,1$; or 
        \item $k=2$ and $C[x,y]$ and $C[y,x]$ each contain at least one insertion point; or 
        \item $k=3$ and $y_1,y_2\in C[x,y]$ are adjacent and $y_3\in C[y,x]$; or 
        \item $k=4$ and $y_1,y_2\in C[x,y]$ are adjacent and $y_3,y_4\in C[y,x]$ are adjacent. 
    \end{enumerate}
\end{Proposition}

\begin{proof}
    Let $Y$ denote the set of insertion points of $G$.
    By Proposition~\ref{prop:spokes}, $G[Y]$ is either $K_1, K_2, 2K_1, K_2\cup K_1, P_2, C_3, 2K_2, P_4, C_4$. 
    Notice that $G[Y] = C_3$ is impossible when $\N \subseteq R$. 
    We will use the following conventions. 
    If $G[Y]$ has an edge, then $y_1$ and $y_2$ are adjacent.
    If $G[Y]$ is a path, then the $y_i$s are indexed in their path order. 

    Notice that Conditions 1.-3. in Theorem~\ref{thm:non-emptycore} are satisfied for $G^+$. 
    Each statement in this proposition is necessary and sufficient for $G^+$ to satisfy Condition 4. of Theorem~\ref{thm:non-emptycore} where $V(C^1)=C[x,y]$ and $V(C^2) = C[y,x]$ in $G^+$.
\end{proof}

Next, we consider the case where $G$ has distinct insertion points, $G[R]$ is a cycle, and at least one vertex in $\N$ is not in the core $R$.

\begin{Proposition}\label{prop:cycle.mess2}
    Suppose $G$ is a graph of two parallel paths with a non-empty core $R$ and distinct insertion points such that $G[R]$ is a cycle. 
    Let $P^i$ be the pendant paths of $G$ with insertion points $y_i$ and let $1\leq k \leq 4$ be the number of pendant paths.
    Let $\N=\{x,y\}$ be a pair of non-adjacent vertices in $G$ with $x$ on a pendant path $P^1$. 
    Then $G^+$ is a graph of two parallel paths if and only if $d^*(x,y)=1$ and one of the following holds
    \begin{enumerate}
    \item $k=1,2$; or  
    \item $k=3$, and $y \in N[P^2]$ or $y_2$ is adjacent to $y_3$; or 
    \item $k=4$ and $y\in N[P^2]$.  
    \end{enumerate}
\end{Proposition}

\begin{proof}
    Let $Y$ denote the set of insertion points of $G$. 
    Once again, by Proposition~\ref{prop:spokes}, $G[Y]$ is either $K_1, K_2, 2K_1, K_2\cup K_1, P_2, C_3, 2K_2, P_4, C_4$. 

    By Lemma~\ref{lem:outerplanar.conditions} and Observation~\ref{obs:core.conditions} $d^*(x,y)=1$ is necessary and sufficient for $G^+$ to be outer-planar with a unique outer-cycle and a single interior edge; which takes care of Condition 1. and 2. in Theorem~\ref{thm:non-emptycore}. 
    Statements 1.-3. are necessary and sufficient to ensure that Conditions 4. hold for $G^+$.
\end{proof}

Finally, we consider the case where $G$ does not have distinct insertion points. Note that in this case, $G$ must be an $\alpha$-graph (see Proposition \ref{prop:alpha.case}).

\begin{Proposition}\label{prop:alpha.mess}
    Let $G$ be an $\alpha$-graph, and $N=\{x,y\}$. 
    Then $G^+$ is a graph of two parallel paths if and only if 
    \begin{enumerate}
        \item $x,y$ are core vertices and $x=c_0$, or 
        \item without loss of generality, $x\in N[P^1]$ and $d^*(x,y) = 1$.
    \end{enumerate}
\end{Proposition}

\begin{proof}
   The reverse direction can be verified by inspection. 
   
   Assume $x,y\neq c_0$ are core vertices, then $G^+$ does not have distinct insertion points, and $G^+$ is not an $\alpha$-graph, which implies that $G^+$ is not a graph of two parallel paths by Proposition~\ref{prop:alpha.case}. 
   Thus, we will proceed with the assumption that either $x$ or $y$ is not a core vertex.

   Assume $d^*(x,y) \neq 1$. 
   If $d^*(x,y)>1$, then $G^+$ is not outer-planar by Lemma~\ref{lem:outerplanar.conditions}. 
   If $d^*(x,y)=0$, then $G^+$ does not have an outer-cycle, and therefore, $G^+$ is not a graph of two parallel paths by Lemma~\ref{prop:hamiltoniancycle}. 

   Assume that $d^*(x,y)=1$ but $x,y$ are not in $N[P^1\cup P^2]$.
   Without loss of generality, $x\in P^3$.
   Then $G^+$ has an interior edge that does not contain $y_1$. 
   Thus, $G^+$ is not an $\alpha$-graph, and consequently, $G^+$ is not a graph of two parallel paths by Proposition~\ref{prop:alpha.case}. 
\end{proof}

\section{Extreme values for $\mr(G^{\N})$}\label{sec:extreme}

In this section, we characterize extreme values for the minimum rank of a probe graph. We start by considering low values. 

\begin{Observation}
Let $G^{\N}$ be a probe graph. Then $\mr(G^{\N})=0$ if and only $G$ is the empty graph. 
\end{Observation}

\begin{Theorem}
Let $G^{\N}$ be a probe graph on $n$ vertices. Then $\mr(G^{\N})=1$ if and only if
%the graph $G$ has at least one edge, 
$G= (K_a\vee \overline{K_b}) \cup \overline{K_c}$ where $a+b+c=n$, $a>0$, and if $b\geq 2$, then $V(K_b)\subseteq \N$. 
\end{Theorem}

\begin{proof}
Beginning with the backward direction, suppose $G$ has the desired form. Notice that since $V(K_b)\subseteq \N$ or $b=1$,\[
\begin{bmatrix}
J_{a\times a}& J_{a\times b} & O_{a\times c}\\
J_{b\times a} &J_{b\times b} & O_{b\times c}\\
O_{c\times a} & O_{c\times b} & O_{c\times c}
\end{bmatrix}\in \cals(G^{\N}).\]
This matrix has rank $1$, which shows that $\mr(G^{\N})\leq 1$. 
Furthermore, since $G$ has at least one edge, $0<\mr(G^{\N})\leq 1$.

For the sake of proving the forward direction, assume that $\mr(G^{\N})=1$. 
This implies that $G$ contains at least one edge. 
Let $P$ be the part of $G^{\N}$ corresponding to the probe vertices. 
By Theorem~\ref{thm.tophalf}, we have $\mr(P)=1$. 
Let $[A\,\,B]$ be a matrix that realizes $\mr(P)=1$ where the columns of $B$ are indexed by $\N$. 
Let $\bx$ be a non-zero column of $[A\,\,B]$, and notice that all columns of $[A\,\,B]$ are scalar multiples of $\bx$. 

For the sake of contradiction, assume that $A=O$ and that $\bx$ has at least two non-zero entries. 
Then $G=(\overline{K_a}\vee \overline{K_b})\cup \overline{K_c}$ where $a$ is the number of non-zero entries in $\bx$ and $b$ is the number of non-zero columns in $B$.
Notice $a\geq 2$ and $b\geq 1$. 
By assumption,  $\N$ does not intersect  $V(\overline{K}_a)$. 
Let $u,w\in V(\overline{K}_a)$, $v\in V(\overline{K}_b)$, and notice that $V\setminus \{v,w\}$ is a forcing set for $G^{\N}$. 
Thus, $\M(G^{\N})\leq n-2$, and so $\mr(G^{\N})\geq 2$, which is a contradiction. 
Therefore, either $A\neq O$ or $\bx$ has exactly one non-zero entry.

\textbf{Case 1:} Suppose $A\neq O$.
This implies that $A=\bx^\top \bx$ since $A$ is symmetric and has rank $1$. 
Let $a$ be the number of non-zero entries in $\bx$, and let $b$ be the number of non-zero columns of $B$, and $n-a-b=c.$
By inspection, 
\[\begin{bmatrix}A&B\\B^\top & O\end{bmatrix} \in \cals(G)\]
where $G=(K_a\vee \overline{K_b}) \cup \overline{K_c}$. 

\textbf{Case 2:} Suppose $\bx$ has exactly one non-zero entry.
This implies that $A$ does not have an off-diagonal non-zero entry, and $G$ is a star. 
In particular, if $b$ is the number of non-zero columns in $B$ and $n-b-1=c$, then $G= (K_1\vee \overline{K_b})\cup \overline{K_c}$.

For either case, suppose $b\geq 2$ and there exists $u\in V(K_b)\setminus \N$ and $u,v\in V(K_b)$. 
Then $G[V(K_a)\cup\{u,v\}]$ has minimum rank $2$, which contradicts the assumption that $G^{\N}$ has minimum rank $1$. 
\end{proof}

In order to characterize probe graphs for which the minimum rank is at most 2, we first recall the following characterization for graphs with minimum rank at most 2.
\begin{Theorem}\cite{MinRank2}\label{thm:mr2Comp}
A graph $G$ has $\mr(G)\leq 2$ if and only if the complement of $G$ is of the form
$$(K_{s_1} \cup K_{s_2} \cup K_{p_1,q_1} \cup\cdots \cup K_{p_k,q_k} ) \vee K_r$$
for appropriate nonnegative integers $k, s_1, s_2, p_1, q_1,\dots, p_k, q_k, r$ with $p_i + q_i >0$ for all $i =1,\dots, k$.
\end{Theorem}

% \begin{Theorem}\cite{MinRank2}\label{thm:mr2Forbidden}
% A connected graph $G$ has $\mr(G)\leq 2$ if and only if G does not contain as
% an induced subgraph any of $P_4$, Dart, Semi-product sign, or $K_{3,3,3}$ (the complete tripartite graph), all shown
% in Fig.
% \end{Theorem}

% \begin{center}
%     \includegraphics[scale=.8]{Screen Shot 2022-09-23 at 1.41.41 PM.png}
% \end{center}
% \textcolor{red}{JK: We want to describe $\N\subseteq V$ containing a non-edge for complement in 5.8 ($\overline{H}$ in proof), such that adding all edges in $\N$ gives a graph that cannot be described as in 5.8 ($\overline{G}$ in proof).}
Relying on Theorem~\ref{thm:mr2Comp}, we now describe a characterization for probe graphs with minimum rank at most 2.

\begin{Theorem}
A probe graph $G^{\N}$ has $\mr(G^{\N})\leq 2$ if and only if the complement of $G$ has vertex set $V(H)$ and edge set $E(\overline{H})+S$, where $H$ is a graph whose complement has the form described in Theorem~\ref{thm:mr2Comp} and $S$ is a subset of the edges in the complete graph $\overline{N}$.
\end{Theorem}

\begin{proof}
First, suppose that $G^{\N}$ is a probe graph such that $\overline{G}$ has vertex set $V(H)$ and edge set $E(\overline{H})+S$, where $H$ is a graph whose complement has the form described in Theorem~\ref{thm:mr2Comp} and $S$ is a subset of the edges in the complete graph $\overline{N}$. Removing the edges $S$ from $\overline{G}$ corresponds to adding these edges to $G$. Since $S$ is contained in $\overline{N}$, the graph $G$ with the edges in $S$ is equivalent to adding edges within $\N$ to $G^{\N}$, and the result of this edge addition is the graph $H$. Therefore, since $\mr(H)\leq 2$, we have $\mr(G^{\N})\leq 2$.
% First suppose $G$ is a graph such that $\overline{G}$ is of the form $(K_{s_1} \cup K_{s_2} \cup K_{s_3} \cup K_{p_1,q_1} \cup\cdots \cup K_{p_k,q_k} ) \vee K_r$
% for appropriate nonnegative integers $k, s_1, s_2, s_3, p_1, q_1,\dots, p_k, q_k, r$ with $p_i + q_i >0$ for all $i =1,\dots, k$. If $s_3=0$, then $\overline{G}$ is of the form described in Theorem~\ref{thm:mr2Comp}, so $\mr(G^{\N})\leq \mr(G)\leq 2$. Otherwise, $s_3>0$ and we assume the non-probe vertices $\N$ are contained in $K_{s_3}$ in $\overline{G}$. In this case, edges can be added to $G$ between vertices in $\N$ such that in $\overline{G}$, $K_{s_3}$ becomes a complete bipartite graph $K_{p_{k+1},q_{k+1}}$. Thus, for this choice of edges added to $G^{\N}$, the complement will have the form described in Theorem~\ref{thm:mr2Comp}, so $\mr(G^{\N})\leq 2$.

Now, suppose $\mr(G^{\N})\leq 2$. If $G$ is a graph such that $\mr(G)\leq 2$, then $\overline{G}$ has the form described by Theorem~\ref{thm:mr2Comp}. This is the case where $S$ is the empty set.
% Since no edges need to be added between non-probe vertices in order to achieve the desired rank, $\N$ may contain any independent set in $G$, which corresponds to any clique in $\overline{G}$. 
Next, suppose $G$ is a graph for which $\mr(G)>2$. Given the set of non-probe vertices $\N$, there must be a set of edges whose addition to $G^{\N}$ forces $\overline{G^{\N}}$ to be of the form described in Theorem~\ref{thm:mr2Comp}. Let one such resulting graph be called $H$. Working backward, there must be a way to delete edges from $H$ to obtain $G$, and this corresponds to adding edges within $\overline{H}$ to obtain $\overline{G}$. Recall that since $\N$ must be an independent set in $G$, these vertices must form a clique in $\overline{G}$. This means that adding edges to a subgraph of $\overline{H}$ must result in that subgraph becoming a clique. Therefore, the edge set of $\overline{G}$ must be $E(\overline{H})+S$, where $S$ is a subset of the edges in the complete graph $\overline{N}$.\end{proof}

We now consider the case of probe graphs with high minimum rank. Observe that since $\mr(G)\leq n-1$ for all graphs $G$, the same holds for the probe graphs. First, we characterize probe graphs with minimum rank $n-1$. Note that in this case, the problem reduces to the usual minimum rank since this can only occur when $N=1$.

\begin{Proposition}
Let $G^{\N}$ be a probe graph on $n$ vertices. Then $\mr(G^{\N})=n-1$ if and only if $|\N|\leq 1$ and $G=P_n$.
\end{Proposition}

\begin{proof}
It is well known that for a graph $G$, $\mr(G)=n-1$ if and only if $G=P_n$, as stated in Theorem~\ref{PathMaxnull1}. If $|\N|\leq 1$ then $\mr(G^{\N})= \mr(G)$, so if $G=P_n$, then $\mr(G^{\N})=n-1$. If $\mr(G^{\N})=n-1$, then by Proposition~\ref{prop.gammabound} we have $|\N|\leq 1$, which implies that $\mr(G)=\mr(G^{\N})=n-1$, so $G=P_n$. 
\end{proof}

For the remainder of the section, we work towards a characterization of probe graphs for which the minimum rank is $n-2$. In the case of the minimum rank of graphs, the following is known. 

\begin{Theorem}[Theorem 5.1 in \cite{johnson2009graphs}]\label{Maxnull2}
Let $G$ be a graph. Then $\mr(G) =n- 2$ if and only if $G$ is a graph of two
parallel paths (but $G\neq P_n$) or $G$ is one of the types shown in Figure~\ref{SpecialGraphswithM2labeled}.
\end{Theorem}

\begin{figure}[h]\label{SpecialGraphswithM2labeled}
\centering
\includegraphics[scale=.5]{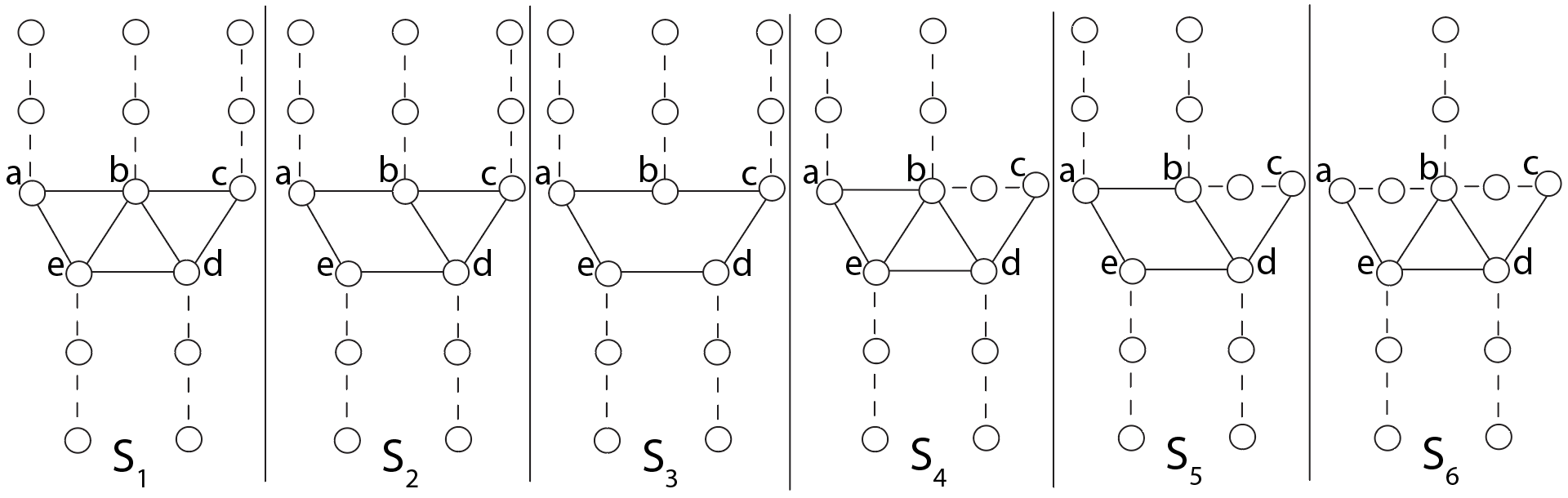}
\caption{Special Graphs}
\end{figure}

We require some additional definitions.
A graph $H$ is a {\em homeomorph} of a graph $G$ if $H$ may be obtained from $G$ by a sequence of edge subdivisions. 
We use $hK_4$ and $hK_{2,3}$ to denote graphs that are homeomorphs of $K_4$ and $K_{2,3}$. 
The graphs $K_4$ and $K_{2,3}$ are considered homeomorphs of themselves.
%Figure~\ref{fig:hK4} shows the general form of $hK_4$ with independent paths $i-j$, $j-k$, $k-t$, $t-i$, $i-k$, and $j-t$. We call the vertices $i$, $j$, $k$, $t$ ``corners" of the structure.

%\begin{figure}[h!]
%\begin{center}
%\includegraphics[height = 3.5cm, width = 4cm]{genhK4}
%\label{fig:hK4}
%\caption{The homeomorph $hK_4$ where the dashed edges represent paths of any length.}
%\end{center}
%\end{figure} 
%\begin{Theorem}[Johnson, Lowery, Smith]\label{Mnull2ifparalelorspec}
 %The graph $G$ satisfies $M(G) = 2$ if and only if $G$ is a graph of two parallel
%paths or $G$ is one of the types listed in Figure~\ref{SpecialGraphswithM2labeled}.
%\end{Theorem}
Before stating the characterization, we require several lemmas. Applying the following known Lemma to two specific graphs leads to Remark~\ref{rem:K4K23}.

\begin{Lemma}[Lemma 2.1 in \cite{johnson2009graphs}]\label{edgesubdlemm}
 Let $G'$ be the graph resulting from an edge subdivision in the graph $G$. Then $\M(G') = \M(G)$ or $\M(G') = \M(G) + 1$.   
\end{Lemma}

\begin{Remark}\label{rem:K4K23}
    Recall $\M(K_4)=M(K_{2,3})=3$. Then, by Lemma~\ref{edgesubdlemm}, any graph $G$ which contains either an $hK_4$ or an $hK_{2,3}$ subgraph satisfies $\M(G) \geq 3$; in particular, $\mr(G)>|V(G)|-2$.
\end{Remark}

Next, we prove two lemmas which we will need for our characterization.

\begin{Lemma}\label{specialgraphcase} 
Let $|\N| = 2$. If $S_i$ be a graph on $n$ vertices where $i \in \{1,2,3,4,5,6\}$ as in Figure~\ref{SpecialGraphswithM2labeled}.
\[\begin{cases}
\mr(S_i^{\N}) = n-2 &\text{ if } i = 2 
\text{ and }\N = \{b,e\}\\
& \text{ if } i = 3 \text{ and }\N = \{b,d\}\\
&\text{ if } i = 5 \text{ and } \N = \{b,e\}\\
\mr(S_i^{\N}) <n- 2 & \text{otherwise}
\end{cases}\]

\end{Lemma}

\begin{proof}
Because of the result of Theorem~\ref{Maxnull2}, $\mr(G) =n-2$ if and only if $G$ is a graph of parallel paths or one of the types in Figure~\ref{SpecialGraphswithM2labeled}. Adding an edge between two nonadjacent vertices in $S_i$ will ensure that the minimum rank of $S_i^{\N}$ is $n-2$ if and only if $S_i^{\N}$ is a graph of parallel paths or remains one of the $S_i$ graphs. This only occurs in the case that $\{be\}$ is added to $S_2$ which yields $S_1$; $\{bd\}$ is added to $S_3$ which yields $S_2$; or $\{be\}$ is added to $S_5$ which yields $S_4$. In any other case, if an edge is added between two vertices in $S_i$, we obtain a graph that is not in Figure~\ref{SpecialGraphswithM2labeled} and is not a graph of parallel paths. 
\end{proof}

Finally, we provide our characterization of probe graphs with minimum rank $n-2$.

\begin{Theorem}
For a connected probe graph $G^{\N}$ on $n$ vertices, $\mr(G^{\N})=n-2$ if and only if 
\begin{itemize}
    \item[1.] $|\N|\leq 1$ and $G$ is a graph of parallel paths which is not a path or $G$ is one of the types in Figure~\ref{SpecialGraphswithM2labeled}; or
    \item[2.] $G$ is a graph of two parallel paths with a non-empty core and $|\N| = 2$ such that $\N$ satisfies the conditions in Proposition~\ref{prop:mess1}, \ref{prop:mess2}, \ref{prop:cycle.mess1}, \ref{prop:cycle.mess2}, or  \ref{prop:alpha.mess}; or
    \item[3.] $G$ is a tree of two parallel paths, $|\N| = 2$, and $G^{\N}$ satisfies the conditions in Proposition~\ref{prop:tree.uniquness}; or
    \item[4.] $G$ is a graph $S_2$, $S_3$, or $S_5$ from Figure~\ref{SpecialGraphswithM2labeled} and $\N=\{b,e\}$, $\N=\{b,d\}$, or $\N=\{b,e\}$, respectively.
\end{itemize}
\end{Theorem}
\begin{proof}
By Proposition~\ref{prop.gammabound}, $2 = \M(G^{\N})\geq |\N|$ and so $2 \geq |\N|$.

If $|\N|\leq 1$, then $\mr(G^{\N}) = \mr(G)$. By Theorem~\ref{Maxnull2}, $\mr(G^{\N})=n-2$ if and only if $G$ is a graph of parallel paths ($G\not=P_n$) or $G$ is one of the types in Figure~\ref{SpecialGraphswithM2labeled}.

If $|\N| = 2$, we only need to consider $G$ and $G^+$, where $G^+$ is the graph resulting from adding the edge between vertices in $\N$.
Thus, it is clear that $\mr(G^{\N})=n-2$ if and only if $\mr(G)\geq n-2$ and $\mr(G^+)\geq n-2$, with at least one bound achieving equality. 
By Theorems~\ref{Maxnull2} and~\ref{PathMaxnull1}, $\mr(G)\geq n-2$ if and only if $G$ is a graph of parallel paths, one of the types in Figure~\ref{SpecialGraphswithM2labeled}, or a path graph.
Observe that it is impossible to add an edge to a graph of parallel paths, one of the types in Figure~\ref{SpecialGraphswithM2labeled}, or a path graph and have the resulting graph be a path graph. 
Therefore, $\mr(G^+)\leq n-2$, since the path is the only graph with minimum rank $n-1$. 
So we see $\mr(G^+)=n-2$ if and only if $G^+$ is either a graph of parallel paths or one of the types in Figure~\ref{SpecialGraphswithM2labeled}. 
We now consider cases to describe when this happens.

Suppose $G$ is a graph of two parallel paths with a non-empty core. By Proposition~Proposition~\ref{prop:mess1}, \ref{prop:mess2}, \ref{prop:cycle.mess1}, \ref{prop:cycle.mess2}, or  \ref{prop:alpha.mess}, the graph $G^+$ is a graph of two parallel paths if and only if $\N$ meets the conditions in the proposition.

Next, suppose $G$ is a graph of two parallel paths with an empty core. This is the same as $G$ being a tree of two parallel paths (note, this case includes the case where $G$ is a path). By Proposition~\ref{prop:tree.uniquness}, the graph $G^+$ is a graph of two parallel paths if and only if $\N$ satisfies the conditions in the proposition.

Finally, suppose $G$ is a graph of the types in Figure~\ref{SpecialGraphswithM2labeled}. By Lemma~\ref{specialgraphcase}, $G^+$ remains a graph in Figure~\ref{SpecialGraphswithM2labeled} if and only if $G$ is a graph $S_2$, $S_3$, or $S_5$ from Figure~\ref{SpecialGraphswithM2labeled} and $\N=\{b,e\}$, $\N=\{b,d\}$, or $\N=\{b,e\}$, respectively.

% By Corollary~\ref{EqqofZandMparallelpaths}, if $G$ is a graph of parallel paths and $\N$ contains exactly two vertices with one on each path, then $\Z(G^{\N}) = \M(G^{\N}) = |\N| = 2$. 
% By Theorem~\ref{thm:PCequalsGamma}, since $\Z(P_n^{\N}) = 2$ whenever $|\N| = 2$, all path graphs have $\M(G^{\N}) = 2$.
% By Lemma~\ref{specialgraphcase} and Lemma~\ref{parallelpathsgamma2withM2conditions}, if $G$ is a graph of parallel paths or a graph of type in Figure~\ref{SpecialGraphswithM2labeled} that satisfies the conditions of these results, then $\M(G^{\N}) = 2$.
\end{proof}

\section*{Acknowledgements}
J.K. is supported by the National Science Foundation through NSF RTG grant DMS-2136228.

\bibliographystyle{plainurl}
\bibliography{bibliography}

\end{document}